\documentclass{amsart}
\usepackage{amsmath}
\usepackage{amsthm,amssymb}
\usepackage{enumerate}
\usepackage[utf8]{inputenc}
\usepackage{color}
\theoremstyle{plain}
\newtheorem{theorem}{Theorem}[section]
\newtheorem{lemma}[theorem]{Lemma}
\newtheorem{prop}[theorem]{Proposition}

\theoremstyle{definition}

\newtheorem{conjecture}[theorem]{Conjecture}

\newcommand{\GL}{\mbox{\rm GL}}
\newcommand{\SL}{\mbox{\rm SL}}
\newcommand{\M}{\mbox{\rm M}}

\newcommand{\Char}{\mbox{\rm char}}

\newcommand{\N}{\mathbb N}
\newcommand{\Z}{\mathbb{Z}}

\newcommand{\Q}{\mathbb{Q}}

\begin{document}
	\date{}
	\title[Generalized power central group identities]{Generalized power central group identities in almost subnormal subgroups of $\GL_n(D)$}
	
	\author[Bui Xuan Hai]{Bui Xuan Hai}
	\address{Faculty of Mathematics and Computer Science, VNUHCM-University of Science,	227 Nguyen Van Cu Str., Dist. 5, HCM-City, Vietnam}
	\email{bxhai@hcmus.edu.vn; huynhvietkhanh@gmail.com; mhbien@hcmus.edu.vn}
	
	\author[Huynh Viet Khanh]{Huynh Viet Khanh}
		
	\author[Mai Hoang Bien]{Mai Hoang Bien}\thanks{The first and third authors are funded by Vietnam National Foundation for Science and Technology Development (NAFOSTED) under Grant No. 101.04-2016.18.}
			
	
	\maketitle
	\begin{abstract} In this paper, we study almost subnormal subgroups of the general linear group $\GL_n(D)$ of degree $n\ge 1$ over a division ring $D$ that satisfy a generalized power central group identity.
	\end{abstract} 

\section{Introduction and preliminaries}\label{ss8.1}
For a ring $R$ with identity, denote by $R^*$ the unit group of $R$. Let $X=\{x_1,x_2,\dots,x_m\}$ be a set of $m$ non-commuting indeterminates. The notation $X^{-1}$ is understood to be the set $\{x_1^{-1}, x_2^{-1},\dots, x_m^{-1}\}$, where $x_i\mapsto x_i^{-1}$ defines a bijection between the sets $X$ and $X^{-1}$. For a field $F$, let $F[X,X^{-1}]$ be the group $F$-algebra of the free group generated by $X$, that is, $F[X,X^{-1}]$ is the ring of all elements having the form of finite sums of $ax_{i_1}^{n_1}x_{i_2}^{n_2}\cdots x_{i_t}^{n_t}$, where $n_i\in \N, a\in F, x_{i_j}\in X\cup X^{-1}$. Some authors call $F[X,X^{-1}]$ the {\it Laurent polynomial ring} on $X$ over $F$. Let $A$ be a ring whose center is $F$. Denote by $F\langle X\rangle$ and $A\langle X\rangle = A*_F F\langle X\rangle$ respectively the free  $F$-algebra on $X$ and the free product of $A$ and $F\langle X\rangle$ over $F$. If $A=D$ is a division ring, then $D\langle X\rangle$ is an fir (free ideal ring), so  $D\langle X\rangle$ has the (universal) division ring of fractions denoted by $D(X)$ (see \cite[Chapter 7]{cohn}). The free product $A*_FF[X,X^{-1}]$ of $A$ and $F[X,X^{-1}]$ over $F$ is called the \textit{generalized Laurent polynomial ring} of $A$ on $X$ over $F$ and is denoted by $A_F[X,X^{-1}]$. Let $f\in A_F[X,X^{-1}]$, that is, $f(x_1,x_2,\dots,x_m)=\sum\limits_{\lambda=1}^{\ell}f_\lambda$  with $$f_\lambda(x_1,x_2,\dots,x_m)=a_{\lambda, 1}x_{i_{\lambda,1}}^{n_{\lambda,1}}a_{\lambda, 2}x_{i_{\lambda,2}}^{n_{\lambda,2}}a_{\lambda, 3}\cdots a_{\lambda, t_\lambda}x_{i_{\lambda,t}}^{n_{\lambda,t}}a_{\lambda,t_\lambda+1},$$ 
where $a_{i,j}\in A, n_{i,j}\in \Z$. If all $n_{i,j}$ are non-negative, then $f$ is called a {\it generalized polynomial} of $A$ over $F$. Assume that $w$ is an element in $A_F[X,X^{-1}]\backslash A$ of the  form $$w(x_1,x_2, \dots, x_m)=a_1x_{i_1}^{n_1}a_2x_{i_2}^{n_2}\cdots a_tx_{i_t}^{n_t}a_{t+1},$$ where $a_i\in A$, $n_i\in \Z$ and $x_{i_j}\in X$. If $a_i\in A^*$ for all $i$, then we call $w$ {\it a generalized group monomial} over $A^*$. 

From now on in this paper, we consider only the particular case when  $A=\M_n(D)$, so $A^*=\GL_n(D)$. In this case, we remark that the following lemma shows that our definition  satisfies the requirements in the definition of generalized group monomials given in \cite{golubchik} and \cite{Pa_To_85} in most cases. In particular, it is the case for division rings with infinite centers we consider in this paper.  So, we would emphasize that in the proofs of our results, it is correct to use the results from \cite{golubchik} and \cite{Pa_To_85}.  Recall that for a group $G$, a subgroup $H$ of $G$ is said to be {\it almost subnormal} in $G$ if there is a family of subgroups 
$$H=H_r\le H_{r-1}\le\cdots\le H_1=G$$
of $G$ such that for each $1<i\le r$, $H_{i}$ is normal in $H_{i-1}$ or $H_i$ has finite index in $H_{i-1}$. We call such a series of subgroups an {\it almost normal series} in $G$.

\begin{lemma}\label{lem:1.1} Let $D$ be a infinite  division ring with center $F$ and $H$ a non-central almost subnormal subgroup of $\GL_n(D)$. 
	
	(i) If $n\geq 2$, then $Z(H)\subseteq F$. 
	
	(ii) If $n=1$, then $Z(H)\subseteq F$ provided $F$ is infinite.
	
\end{lemma}   
\begin{proof}
	(i) By \cite[Theorem 3.3]{Pa_NgBiHa_2016},  $\SL_n(D)\subseteq H$. Consequently, the subring $F[H]$ is normalized by $\GL_n(D)$. In view of the Cartan-Brauer-Hua theorem for matrix rings, it follows that $F[H]=\M_n(D)$, and hence $Z(H)\subseteq F$. 
	
	(ii) Let us  consider the division subring $F(H)$ of $D$ generated by $H$ over $F$. Since $F(H)$ is $H$-invariant, that is, $xF(H)x^{-1}=F(H)$ for every element $x\in H$, by \cite[Theorem 3.10]{Pa_dbh_19}, either  $F(H)=D$ or $F(H)\subseteq F$. Because $H$ is non-central, we have  $F(H)=D$, hence it follows that $Z(H)\subseteq F$.  
\end{proof}
Assume that $H$ is a subgroup of $\GL_n(D)$ and $w$ is a generalized group monomial over $\GL_n(D)$. We say that $w=1$ is a \textit{generalized group identity} (briefly, GGI) of $H$ if $w(c_1,c_2,\dots,c_m)=1$ for any $(c_1,c_2,\dots,c_m)\in H^m$. If for every $(c_1,c_2,\dots,c_m)\in H^m$, there is a positive integer $p$ depending on $(c_1,c_2,\dots,c_m)$ such that $w(c_1,c_2,\dots,c_m)^p\in F$, then $w$ is called a \textit{generalized power central group identity} (briefly, GPCGI) of $H$ (see   \cite{chiba_88}). In this case, we also say that $H$ \textit{satisfies the} GPCGI $w$. A GPCGI $u$ of $H$ is  \textit{non-trivial} if it satisfies the condition $u^p\in \M_n(D)_F[X,X^{-1}]\backslash \M_n(D)$ for any positive integer $p$.
A well-known example of a generalized group monomial is  $w(x)=axa^{-1}x^{-1}\in D_F[x, x^{-1}]$, where $a\in D\backslash F$ \cite{her}. Herstein \cite[Conjecture 2]{her} conjectured that $w$ is not a GPCGI of $D^*$. In this paper, we will consider GPCGIs of almost subnormal subgroups of $\GL_n(D)$. Almost subnormal subgroups of skew linear groups were first studied by B. Hartley in \cite{Pa_Ha_89} and recently have received attentions (see also \cite{Pa_dbh_19,  dung_19, hai-khanh_19, Pa_NgBiHa_2016}). In this paper, we mainly consider the following conjecture.
   
\begin{conjecture}\label{c1.2} \textit{Let $D$ be a division ring with center $F$, $\GL_n(D)$  the general linear group of degree $n$ over $D$ such that if $n\ge 2$ then $D$ is assumed not to be a locally finite field,  and  $$w(x_1,x_2, \dots, x_m)=a_1x_{i_1}^{\alpha_1}a_2x_{i_2}^{\alpha_2}\cdots a_tx_{i_t}^{\alpha_t}a_{t+1}$$  a generalized group monomial over $\GL_n(D)$. If $N$ is an almost subnormal subgroup of $\GL_n(D)$ and $w$ is a non-trivial {\rm GPCGI} of $N$, then $N$ is central.} 
\end{conjecture}

Recall that a field $F$ is {\it locally finite} if every finite subset of $F$ generates over its prime subfield $P$ a finite subfield.  Clearly, $F$ is locally finite if and only if $P$ is finite and $F$ is algebraic over $P$. Some authors also call a locally finite field  an {\it absolute field} (e.g. see \cite{Pa_GoSh_09}). Fields of zero characteristic are trivial examples of  fields that are not locally finite. Non-trivial examples of such a kind are  uncountable fields \cite[Corollary B-2.41, p. 343]{rotman}. 

Note that in the case $n\ge 2$, the assumption for $D$ in Conjecture \ref{c1.2} is really necessary because in this case,  the group $\GL_n(F)$ always satisfies a non-trivial GPCGI. To see this, we assume that $n\ge 2$ and $D=F$ is a locally finite field. Then, for every element $a\in \GL_n(F)$, the subfield $P_a$ of $F$ generated by all entries of $a$ over the prime subfield $P$ of $F$ is finite. Moreover, $a\in \GL_n(P_a)$ which implies $a^m=1$, where $m=|\GL_n(P_a)|$. Hence, $\GL_n(F)$ satisfies the GPCGI $w(x)=x$.

In 1978, Herstein \cite[Conjecture 2]{her} conjectured that if every multiplicative commutator in a division ring $D$ is radical over the center $F$ of $D$, then $D$ is a field (recall that an element $x\in D$ is \textit{radical} over $F$ if there exists a positive integer $n(x)$ depending on $x$ such that $x^{n(x)}\in F$, and a subset $S$ of $D$ is radical over $F$ if every its element is radical over $F$). Clearly, this conjecture may be considered as
a particular case of Conjecture \ref{c1.2}. In fact, it is the case if $n=1$, $N=D^*$ and $w(x,y)=xyx^{-1}y^{-1}$. We should note that even this case is not solved in general. Some special cases of Conjecture~\ref{c1.2} were mentioned in several papers such as \cite{Pa_Bi_16,Pa_Bi_15, Pa_BiDu_14,chiba_88, gon-ma_86, Ma_Li_85}.

In the case when $n\ge 2$ and $F$ is infinite, by \cite[Theorem 3.3]{Pa_NgBiHa_2016}, we can assume that $N$ is normal in $\GL_n(D)$ . The answer to  Conjecture \ref{c1.2} is affirmative if $n=1$, $N$ is subnormal in $D^*$ and either $D$ is centrally finite or $F$ is uncountable \cite{Pa_Bi_16}. The aim of this note is to support some positive cases of this conjecture.

In the last section, we consider a special case when all coefficients of the generalized power central group identity are $1$. Actually, in this case, we focus on a general topic: the existence of non-cyclic free subgroups in almost subnormal subgroups of division rings with uncountable center. 

Finally, we briefly observe the relations of some classes of division rings we consider in this paper.  By definition, a centrally finite division ring is locally finite, a locally finite division ring is both weakly locally finite and algebraic over its center. It is well-known in the literature that there exists a various number of locally finite division rings that are not centrally finite. In \cite{Pa_Bi2_15}, Deo et al. have constructed infinitely many weakly locally finite division rings that are not algebraic over their centers. So, in particular, the class of weakly locally finite division rings  strictly contains the class of locally finite division rings. In 1941, Kurosh \cite[Problem K]{kurosh} asked whether an algebraic division ring is locally finite? Up to present, this question remains without the  definite answer and it is now often referred to as the Kurosh Problem for division rings. The following chart reflects the relations between these classes of division rings.

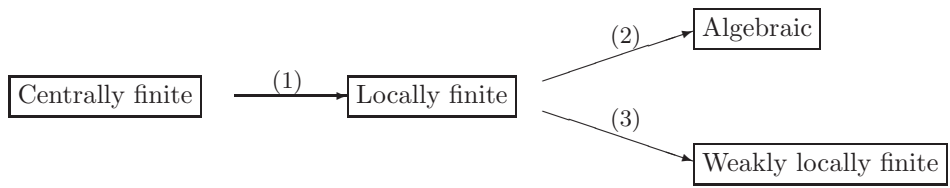
\begin{center}
	\unitlength=1mm
	\special{em:linewidth 0.4pt}
	\linethickness{0.4pt}
	\begin{picture}(120.00,20.00)(0,0)
	\put(0,5){\fbox{Centrally finite}}
	\put(30,6){\vector(1,0){15}}
	\put(35,7){\small{(1)}}
	\put(45,5){\fbox{Locally finite}}
	\put(71,8){\vector(3,1){20}}
	\put(71,4){\vector(3,-1){20}}
	\put(80,13){\small{(2)}}
	\put(80,2){\small{(3)}}
	\put(91,14){\fbox{Algebraic}}
	\put(91,-4){\fbox{Weakly locally finite}}
	\end{picture}
\end{center}
\vspace*{1cm}

In this chart, the arrows mean the inclusions. Note that arrows (1) and (3) are not reversible by corresponding counter-examples, while the question whether arrow (2) is reversible is exactly the Kurosh Problem for division rings.

\section{Almost subnormal subgroups of the general linear group over a weakly locally finite division ring}

In this section, we give the affirmative answer to Conjecture \ref{c1.2} for weakly locally finite division rings. Recall that a division ring $D$ is \textit{weakly locally finite} if every its finite subset generates a centrally finite division subring.
Note that the class of weakly locally finite division rings  strictly contains the class of locally finite division rings (see \cite{Pa_Bi2_15}). We need two auxiliary lemmas whose proofs  are similar to the proofs of \cite[Lemma 2.1]{Pa_Bi_16} and \cite[Proposition 2.3]{Pa_Bi_16}, so they can be omitted. 

\begin{lemma}\label{l5.2} Let $D$ be a division ring with center $F$,  $w(x_1,x_2,\dots,x_m)$  a generalized group monomial over ${\rm GL}_n(D)$. Assume that $H$ is a subgroup of ${\rm GL}_n(D)$ such that $w$ is a {\rm GPCGI} of $H$. Then, $w$ is non-trivial if and only if $w^p\not\in \M_n(D)$ for any positive integer $p$, that is, $w^p$ is a generalized group monomial for any positive integer $p$.
\end{lemma}

\begin{lemma}\label{l5.4} Let $D$ be a division ring, and $H$ a subgroup of ${\rm GL}_n(D)$. If $H$ satisfies a non-trivial {\rm GPCGI}, then $H$ satisfies a {\rm GPCGI} of the form $$a_1x_{i_1}^{n_1}a_2x_{i_2}^{n_2}\dots a_tx_{i_t}^{n_t}a_{t+1},$$ where $i_1\ne i_{t}$.
\end{lemma}

Let $D$ be a division ring with center $F$ and $N$ an almost subnormal subgroup of $\GL_n(D)$ with almost normal series
$$N=N_r\leq N_{r-1}\leq \cdots\leq N_1=\GL_n(D).$$ For any $a\in N$, we construct a sequence of elements $c_i(a,x)$ in $\M_n(D)_F[x,x^{-1}]$ by induction on $i$ as the following.

Set $c_1(a,x)=axa^{-1}x^{-1}$. Further, for any $i, 1<i\leq r$, set $$c_i(a,x)=c_{i-1}(a,x)ac_{i-1}(a,x)^{-1}a$$if $N_i$ is normal in $N_{i-1}$ and $$c_i(a,x)=c_{i-1}(a,x)^{\ell_i}$$ if $N_i$ has finite index in $N_{i-1}$ with $\ell_i=[N_{i-1}:N_i] !$.

\begin{lemma}\label{l7.1}
	Let $N_1,N_2,\dots,N_r$ and $c_1(a,x), c_2(a,x),\dots, c_r(a,x)$ be as above.  Then
	\begin{enumerate}
		\item The monomial $c_i(a,x)$ is represented in the form 
		$$c_i(a,x)=a^{n_1}x^{m_1} a^{n_2}x^{m_2}\cdots a^{n_t}x^{m_t}, \eqno(*)$$
		where $n_j,m_j\in \{1,-1\}$ and $n_1=1$;
		\item if $a\in \GL_n(D)\backslash F$, then $c_i(a,x)$ and $u_i(a,x)=a^{-1}c_i(a,x)$ are generalized group monomials over $\GL_n(D)$ for every $i\ge 1$;
		\item we have $c_i(a,b), u_i(a,b)\in N_i$ for every $b\in \GL_n(D)$.
	\end{enumerate}
\end{lemma}
\begin{proof}
	We prove (1) by induction on $i$.  By definition, $c_1(a,x)=axa^{-1}x^{-1}$ has the form $(*)$. Assume that $i>1$ and $c_{i-1}$ has the form $(*)$, that is, 
	$$c_{i-1}(a,x)=a^{n_1}x^{m_1} a^{n_2}x^{m_2}\cdots a^{n_t}x^{m_t}, $$
	where $n_j,m_j\in \{1,-1\}$ and $n_1=1$. We will show that $c_i(a,x)$ has the form $(*)$ too. Indeed, there are two cases to examine.
	
	\begin{enumerate}[(i)]
		\item If $N_i$ is normal in $N_{i-1}$, then 
		
		$c_i(a,x)=c_{i-1}(a,x)ac_{i-1}(a,x)^{-1}a$
		
		$=(ax^{m_1} a^{n_2}x^{m_2}\cdots a^{n_t}x^{m_t})a (x^{-m_t}a^{-n_t}x^{-m_{t-1}} a^{-n_{t-1}}\cdots x^{-m_1}a^{-1})a$
		
		$=ax^{m_1} a^{n_2}x^{m_2}\cdots a^{n_t}x^{m_t}a x^{-m_t}a^{-n_t}x^{-m_{t-1}} a^{-n_{t-1}}\cdots x^{-m_1}$.
		
		\item If $N_i$ has finite index in $N_{i-1}$ with $\ell_i=[N_{i-1}:N_i]!$, then 
		
		$c_i(a,x)=c_{i-1}(a,x)^{\ell_i}$
		
		$=(ax^{m_1} a^{n_2}x^{m_2}\cdots a^{n_t}x^{m_t})^{\ell_i}$
		
		$=ax^{m_1} a^{n_2}x^{m_2}\cdots a^{n_t} x^{m_t}\cdots ax^{m_1} a^{n_2}x^{m_2}\cdots a^{n_t} x^{m_t}$.
	\end{enumerate}
	
	We see that in both cases  $c_i(a,x)$ has the form $(*)$, so the proof of (1) is complete. The proofs of (2) and (3) follow directly from (1).
\end{proof}

Using these three lemmas, we can prove the following very useful result.
\begin{lemma}\label{lem:2.3} Let $D$ be a division ring with infinite center $F$ and $N$ a non-central almost subnormal subgroup of ${\rm GL}_n(D)$. If $N$ satisfies some non-trivial {\rm GPCGI}, then so does ${\rm GL}_n(D)$. 
\end{lemma}
\begin{proof}
Let $$w(x_1,x_2,\dots, x_m)=a_1x_{i_1}^{\alpha_1}a_2x_{i_2}^{\alpha_2}\cdots a_tx_{i_t}^{\alpha_t}a_{t+1}$$ be a non-trivial GPCGI of $N$. By Lemma~\ref{l5.4}, we can assume  $i_1\ne i_t$.  Replacing $x_i$ by  $x_ix_{m+i}$, where $x_{m+i}$ is an indeterminate, we have $$w(x_1,x_2,\dots, x_m,x_{m+1},\dots,x_{2m})=b_1x_{j_1}^{\beta_1}b_2x_{j_2}^{\beta_2}\cdots b_sx_{j_s}^{\beta_s}b_{s+1},$$ where $\beta_j\in \{1,-1\}$. 
Let $a\in N\backslash F, y_1, y_2, \dots, y_m$ be $m$ indeterminates, and  $u_r(a,y)=a^{-1}c_r(a,y)$  as in Lemma~\ref{l7.1}. Then, by Lemma~\ref{l7.1}, $u_r(a,b)\in N$ for every $b\in \GL_n(D)$, so  if $$w'(y_1,y_2,\dots, y_m)=w(u_r(a,y_1),u_r(a,y_2),\dots, u_r(a,y_m)),$$ then $w'$ is a GPCGI of $\GL_n(D)$. Moreover, since $u_r(a,y_i)$ is a generalized group monomial over $\GL_n(D)$, and both the beginning and end of $u_r(a,x)$ are $y_i$, one has $w'\not\in \GL_n(D)$. Observe that $i_1\ne i_t$, so  $w'^p\not \in \GL_n(D)$ for any integer $p\ge 1$. 
Therefore, $w'$ is a non-trivial GPCGI of $\GL_n(D)$ by Lemma~\ref{l5.2}. The proof of the  lemma is now complete.
\end{proof}

In the remaining part of this section, we assume that $D$ is a weakly locally finite division ring.
\begin{lemma}\label{l5.1}
	Let $D$ be a weakly locally finite division ring with center $F$ and let $P$ be the prime subfield of $F$. If $S$ is a finite subset of $D$, then there exists a centrally finite division subring $D_1$ of $D$ containing $S$ such that the center $F_1=Z(D_1)$ is finitely generated over $P$. Additionally, if $D$ is not a locally finite field, then so is $F_1$.
\end{lemma}
\begin{proof} Let $S$ be a finite subset of $D$. Since $D$ is weakly locally finite, the division subring $D_1=P(S)$ of $D$ is centrally finite. By \cite[Lemma 2.6]{Pa_Bi_16}, $F_1=Z(D_1)$ is finitely generated over $P$.
	
Assume that $D$ is not a locally finite field. If $D$ is non-commutative, then we choose two elements $a, b\in D$ such that $ab\neq ba$.  Replace $D_1=P(S)$ by $D_1=P(S\cup \{a,b\})$  in the previous part, we see that $D_1$ is a non-commutative centrally finite division subring. By Jacobson's theorem (e.g. see \cite[Theorem~13.11]{lam}),  the center $F_1$ of $D_1$ is not locally finite. Now, assume that $D$ is commutative. If $P=\Q$, then $D_1$ is not locally finite. If $P$ is finite, then choose an element $u\in D$ which is not algebraic over $P$, we have $D_1=P(S\cup\{u\})$ is a subfield of $D$ which is not locally finite. 
\end{proof}	 

\begin{theorem} \label{Th2.9}
	Let $D$ be a weakly locally  finite division ring with infinite center $F$ and assume that $N$ is an almost subnormal subgroup of ${\rm GL}_n(D)$. If $n\geq 2$, then we assume in addition that $D$ is not a locally finite field. If $N$ satisfies a non-trivial {\rm GPCGI}, then $N$ is central. 
\end{theorem}
\begin{proof}
	Assume that $N$ is non-central. By Lemma~\ref{lem:2.3}, ${\rm GL}_n(D)$  satisfies a non-trivial GPCGI, say
	$$w(x_1,x_2,\dots, x_m)=a_1x_{i_1}^{n_1}a_2x_{i_2}^{n_2}\cdots a_tx_{i_t}^{n_t}a_{t+1}.$$ Suppose $n=1$. By \cite[Theorem 1]{Pa_Bi2_15}, $D$ is locally PI, so by applying \cite[Theorem~ 2.9]{Pa_Bi_16}, one has $D^*$ is a field, and we are done. Now, assume $n>1$. Let $S$ be the set of all entries of all $a_i$ and $D_1$ be the division subring defined as in Lemma~\ref{l5.1} with respect to $S$. Then, by Lemma~\ref{l5.1}, $D_1$ is a centrally finite division ring whose center $F_1$ is not a locally finite field. Clearly, $w(x_1,\dots,x_m)$ is a generalized group monomial over $\GL_n(D_1)$. Now, for any elements $c_1,c_2,\dots,c_m\in \GL_n(D_1)$, there exists a positive integer $p$ such that $w(c_1,c_2,\dots,c_m)^p\in F\cap \GL_n(D_1)\subseteq F_1$. Hence, in  view of \cite[Lemma 3.1]{her2}, there exists a positive  integer $M$ such that $w(c_1,c_2,\dots,c_m)^M\in F_1$ for all $c_i\in \GL_n(D_1)$. Therefore, if we put $$u(x_1,x_2,\dots,x_m)=w(x_1,x_2,\dots,x_m)^Myw(x_1,x_2,\dots,x_m)^{-M}y^{-1}$$ where $y$ is an indeterminate, then $u=1$ is a GGI of $\GL_n(D_1)$. Since $F_1$ is not locally finite, $F_1$ is infinite. In view of  \cite[theorems 1, 2]{golubchik}, we have $n=1$ and $D_1$ is commutative, which is a contradiction. Thus, $N$ is central.
\end{proof}	

\section{Invertible elements in Laurent polynomial series algebras}\label{s2.1}	

Let $A$ be a $k$-algebra, where $k$ is a field. Recall that an element $a\in A$ is {\it algebraic} over $k$ if there exists a non-zero polynomial $$f(x)=x^n+a_{n-1}x^{n-1}+\cdots+a_0$$ in $k[x]$ such that $f(a)=0$. The $k$-algebra $A$ is {\it algebraic} over $k$ if every element of $A$ is algebraic over $k$.

(i) Let $R$ be a $k$-algebra and assume that $t$ is an indeterminate which commutes with every element of $R$. We denote by $R((t))$ the set of all Laurent series $\sum\limits_{i=n}^\infty a_it^i$, where $n\in \Z, a_i\in R$. Then, $R((t))$ is an algebra whose addition and multiplication are defined as follows: for any two elements $\alpha=\sum\limits_{i=n}^\infty a_it^i,\beta= \sum\limits_{i=m}^\infty b_it^i\in R((t))$,  clearly, we can assume that $m=n$. Then $$\alpha+\beta=\sum\limits_{i=n}^\infty (a_i+b_i)t^i$$ and $$\alpha\beta=\sum\limits_{i=2n}^\infty c_it^i,$$ where $c_i=\sum\limits_{i=u+v}a_u b_v$. Note that if $a_0$ is invertible in $R$, then so is $\alpha=\sum\limits_{i=0}^\infty a_it^i$ in $R((t))$ \cite[Example 1.7]{lam}.   
Moreover, for every $a\in R$, the element $1+at$ is invertible in $R((t))$ and $(1+at)^{-1}=\sum\limits_{i=0}^\infty (-1)^ia^it^i.$ 
\bigskip

(ii) We may write $(1+at)^{-1}$ in another form if $a$ is algebraic over $k$. Indeed, assume that  $f_a(x)=x^n+c_{n-1}x^{n-1}+\cdots+c_0$ is a non-zero polynomial such that $f_a(a)=0$. Put $a_t =1+ at$. Then, $$0=\Big(\frac{a_t-1}{t}\Big)^n+c_{n-1}\Big(\frac{a_t-1}{t}\Big)^{n-1}+\cdots+c_0.$$ After expanding and multiplying two sides by $t^n$, one has $$a_t^n+g_{n-1}(t)a_t^{n-1}+\cdots+g_0(t)=0,$$ where all $g_i(t)$ are suitable polynomials in the polynomial ring $k[t]$. In particular, $g_0(t)=(-1)^n+(-1)^{n-1}c_{n-1}t+\cdots+(-1)^0c_0t^n=t^nf_a(-t^{-1})$ is non-zero in $k[t]$, and consequently, $$a_t^{-1}=(-a_t^{n-1}-g_{n-1}(t)a_{t}^{n-2}-\cdots - g_1(t) )g_0(t)^{-1}.$$ Since $g_0(t)$ is in $k[t]\subseteq k((t))$,  we can write $$(1+at)^{-1}=\frac{-(1+at)^{n-1}-g_{n-1}(t)(1+at)^{n-2}-\cdots - g_1(t) }{g_0(t)}.$$

If $a\in R$ is an algebraic element over $k$, then this form of $(1+at)^{-1}$ is used parallel with the form presented in (i).  

(iii) The representation of invertible elements in (ii) has the following  application: Let $a\in R$ be an algebraic element over $k$ and $\beta\in k$. Assume that $g_0$ is defined as in (ii). Hence, $$g_0(t)=(1+at) (-(1+at)^{n-1}-g_{n-1}(t)(1+at)^{n-2}-\cdots - g_1(t)),$$ which implies $$g_0(\beta)=(1+a\beta) (-(1+a\beta)^{n-1}-g_{n-1}(\beta)(1+a\beta)^{n-2}-\cdots - g_1(\beta)).$$ As a corollary, if $g_0(\beta)\ne 0$, then $1+a\beta$ is invertible in $R$. Consequently, there are only finitely many $\beta\in k$ such that $1+a\beta$ is non-invertible in $R$.

\section{Almost subnormal subgroups of the general linear group over a division ring with uncountable center}

In this section, we study Conjecture \ref{c1.2} for  division rings whose centers are uncountable. The following result extends \cite[Theorem 2.11]{Pa_Bi_16}.
\begin{prop}
\label{c2.7} Let $D$ be a non-commutative division ring with uncountable center $F$ and assume that $N$ is an almost subnormal subgroup of $D^*$. If $N$ satisfies a non-trivial {\rm GPCGI}, then N is contained in $F$.
\end{prop}
	
	\begin{proof} Assume that $N$ is not contained in $F$. By Lemma~\ref{lem:2.3}, $D^*$ also satisfies a non-trivial GPCGI. In view of \cite[Corollary 3]{chiba_88}, $D$ is centrally finite. By Theorem~\ref{Th2.9}, $D^*$ is abelian, a contradiction. 
	\end{proof}
	The following proposition may be considered as a generalization of \cite[Corollary~2]{chiba_88} and \cite[Corollary 2.12]{Pa_Bi_16}.
	
	\begin{prop} Let $D$ be a division ring with uncountable center $F$ and $a\in D^*$. Assume that $N$ is a non-central almost subnormal subgroup of $D^*$.  If for every $x\in N$, there exists an integer $n(x)$ depending on $x$ such that 
	${\left( {ax{a^{ - 1}}{x^{ - 1}}} \right)^{n\left( x \right)}} \in F$, then  $a\in F$.
	\end{prop}
	\begin{proof} If $a$ is not in $F$, then $w\left( x \right) = ax{a^{ - 1}}{x^{ - 1}}$ is a non-trivial GPCGI of $N$. By Proposition~\ref{c2.7}, we conclude that $N$ is central, a contradiction.
	\end{proof}
		In the remaining part of this section, we study  the general linear group of degree $n\geq 2$ over a division ring which is algebraic over its uncountable center. 
\begin{lemma}\label{l3.1} Let $R$ be an algebra over a field $k$ and let $f(t)\in R[t]$. If $f(t)$ has infinitely many roots in $k$, then $f(t)\equiv 0$.\end{lemma}
\begin{proof}
	The proof is the standard Vandermonde argument (or see \cite[propositions 2.3.26 and 2.3.27]{Rowen}).
\end{proof}

\begin{lemma}\label{l3.1} Let $k$ be a field, $R$ a ring whose center is $k$, and  $f(t)\in R[t]$. If $f(\alpha) \in k $ for infinitely many elements  $ \alpha \in k$, then $f(t) \in k[t]$.
\end{lemma}

\begin{proof}For each $g(t) \in R[t]$, set $ h(t)= f(t)g(t)-g(t)f(t) $. By hypothesis, it is clear that $h(\alpha)=0$ for infinitely many $\alpha \in k$. So, in view of the previous lemma, $h(t)\equiv 0$ which implies $f(t) \in k[t]$.
\end{proof}

Now, assume that $D$ is a division ring with infinite center $F$, $A=\M_n(D)$ is the ring of matrices over $D$ of degree $n\ge 1$,  $R=A[y_1,y_2,\dots,y_m]$ is the polynomial ring on non-commuting indeterminates $y_1,\dots,y_m$. It is easy to see that $A$ is commutative if and only if $D=F$ and $n=1$. In the following, we consider only the case when $A$ is non-commutative because otherwise our results should be trivial. 
\begin{lemma}\label{l3.2} 
	Let $w(x_1,x_2,\dots,x_m)=a_1x_{i_1}^{n_1}a_2\cdots a_tx_{i_t}^{n_t}a_{t+1}$ be a generalized group monomial over $A^*=\GL_n(D)$. Then, the element $$w(1+y_1t,1+y_2t,\dots, 1+y_mt)$$ may be written in $R((t))$ in the following form $$w(1+y_1t,1+y_2t, \dots,1+y_mt)=\sum\limits_{i=0}^\infty f_i(y_1,y_2,\dots,y_m)t^i,$$ where $f_i(y_1,y_2,\dots,y_m)$ is a generalized polynomial over $R$ and  homogeneous of degree $i$.   Moreover, $f_0(y_1,y_2,\dots,y_m)=w(1,1,\dots,1)$ and if $w\not\in A^*$, then there exists $i_0\ge 1$ such that $f_{i_0}\not \equiv 0$.
\end{lemma}
\begin{proof}
	Since $(1+y_it)^{-1}=\sum\limits_{j=0}^{\infty} (-1)^jy_i^jt^j$, the fact that $$w(1+y_1t,1+y_2t,\dots,1+y_mt)=\sum\limits_{i=0}^\infty f_i(y_1,y_2,\dots,y_m)t^i,$$ where $f_i(y_1,y_2,\dots,y_m)$ is a generalized polynomial over $R$ and  homogeneous of degree $i$, is clear. It is easy to see that $f_0(y_1,y_2,\dots,y_m)=w(1,1,\dots,1)$. Now, assume that $w\not\in A^*$ and all $f_i\equiv 0$ for  $i\ge 1$. Then, $$w(1+y_1t,1+y_2t,\dots, 1+y_mt)=w(1,1,\dots,1).$$ Observe that $a=w(1,1,\dots,1)=a_1a_2\cdots a_{t+1}\in A^*$, so $$a^{-1}w(1+y_1t,1+y_2t,\dots, 1+y_mt)\equiv 1.$$
	For $c_1,c_2,\dots,c_m\in A^*$, one has $$1=a^{-1}w(1+(c_1-1)1,1+(c_2-1)1,\dots, 1+(c_m-1)1).$$ Hence, $a^{-1}w=1$ is a GGI of $A^*$ because $w\not\in A^*$. By \cite[theorems 1, 2]{golubchik}, $A$ is commutative, a contradiction.
\end{proof}

We are now ready to prove the main result of this section.
\begin{theorem}\label{t3.3} Let $D$ be a division ring that is algebraic over its uncountable center $F$ and assume that $N$ is an almost subnormal subgroup of $\GL_n(D)$. If $N$ satisfies a non-trivial {\rm GPCGI}, then $N$ is central.
	 \end{theorem}
\begin{proof} Put $A=\M_n(D)$. For $n=1$, see Proposition~\ref{c2.7}. Now, assume that  $n\ge 2$ and $N$ is non-central. By \cite[Theorem 3.3]{Pa_NgBiHa_2016},  $N$ is normal in $A^*=\GL_n(D)$. We claim that $D$ is centrally finite. Indeed, if $D$ is commutative, then there is nothing to do. Thus, suppose that $D$ is non-commutative. Since $D$ is algebraic over uncountable center $F$, by \cite[Theorem 2.10]{Pa_AkAr_01}, $A$ is algebraic over $F$ . Now, let $w(x_1,x_2,\dots,x_m)=a_1x_{i_1}^{n_1}a_2\cdots a_tx_{i_t}^{n_t}a_{t+1}$ be a non-trivial GPCGI of $A^*$. Then, there exists a positive integer $q$ such that $w{\left( {1 ,\dots,1 } \right)^q} \in~ F$. Replacing $w$ by  $w^q$ if it is necessary, without loss of generality, we can assume that $w(1,1,\dots,1)=a\in F$.

Let  $Y=\{ y_1,y_2,\dots,y_m\}$ be $m$ indeterminates and $R=A[Y,Y^{-1}]$ be the group ring over $A$ of the  free group generated by $Y$. By using the representation in Section~\ref{s2.1} (i), the element $w(1+y_1t,1+y_2t,\dots,1+y_mt)$ in $R((t))$, after expanding, has the form 
$$w(1+y_1t,1+y_2t,\dots,1+y_mt)=\sum\limits_{i=0}^\infty f_i(y_1,y_2,\dots,y_m)t^i,$$ where $f_i(y_1,y_2,\dots,y_m)\in R$ and $f_i$ is homogeneous of degree $i$. By Lemma~\ref{l3.2}, $f_0(y_1,y_2,\dots,y_m)=a$ and there exists $i_0\ge 1$ such that $f_{i_0}\not \equiv 0$. Write $w$ in the following form $$w(1+y_1t,\dots,1+y_mt)=a+f_{i_0}(y_1,\dots,y_m)t^{i_0}+f_{i_0+1}(y_1,\dots,y_m)t^{i_0+1}+\cdots.$$	
We show that $f_{i_0}(y_1,y_2,\dots,y_m)^\alpha y-yf_{i_0}(y_1,y_2,\dots,y_m)^\alpha=0$ is a generalized polynomial identity of $A$, where $y$ is an indeterminate and   $\alpha$ is a positive integer. To do  this, for any elements  $c_1,c_2,\dots,c_m\in A$, it suffices to prove that there exists a positive integer $\alpha$ such that $f_{i_0}(c_1,c_2,\dots,c_m)^\alpha\in F$. 
Indeed, we have $$w(1+c_1t,\dots,1+c_mt)=a+f_{i_0}(c_1,\dots,c_m)t^{i_0}+f_{i_0+1}(c_1,\dots,c_m)t^{i_0+1}+\cdots.$$	 In view of  Section~\ref{s2.1} (iii),  $1+c_i\lambda$ is not invertible in $A$ only for finitely many  $\lambda\in ~F$. Hence, there are uncountably many $\lambda\in F$ such that $1+c_i\lambda$ are invertible in $A$. Consequently, there exists a positive integer $M$ such that  $$w(1+c_1\lambda,1+c_2\lambda,\dots,1+c_m\lambda)^{M}\in F$$ 
for uncountably many $\lambda\in F$. There are two cases to consider.
\bigskip
	
\noindent
\textit{ Case 1. 	$\Char(F)=0$:} 

\bigskip
We have $$u(t)=w(1+c_1t,\dots,1+c_mt)^M=a^M+Ma^{M-1}f_{i_0}(c_1,\dots,c_m)t^{i_0}+\cdots.$$

\noindent
\textit{ Case 2. $\Char(F)=p>0$:}

\bigskip
Suppose that $M=p^\lambda N$, where $a\in \N$ and $(N,p)=1$. Observe that $$(a+f_{i_0}(c_1,c_2,\dots,c_m)i_0+\cdots)^{p^\lambda}=a^{p^\lambda}+f_{i_0}(c_1,c_2,\dots,c_m)^{p^\lambda}t^{i_0p^\lambda}+\cdots,$$ so $$u(t)=w(1+c_1t,\dots,1+c_mt)^M=a^M+Na^{p^\lambda}a^{N-1} f_{i_0}(c_1,\dots,c_m)^{p^\lambda}t^{i_0p^\lambda}+\cdots.$$ 
	
Hence, in both cases, the coefficients of the first two terms of $$u(t)=w(1+c_1t,1+c_2t,\dots,1+c_mt)^M$$ are $a^M$ and $bf_{i_0}(c_1,\dots,c_m)^{\alpha}$ for some $b\in F$ and $\alpha\in \N$. 
On the other hand, since $c_1,\dots,c_m$ are algebraic over $F$, according to Section~\ref{s2.1} (ii), we may assume that $u(t)=\frac{h_1(t)}{h_2(t)}$, where $h_1(t)\in A[t]$ and $h_2(t)\in F[t]$. Therefore, $\frac{h_1(\lambda)}{h_2(\lambda)}\in F$ for infinitely many $\lambda\in F$. Since $h_2(\lambda)$ belongs to $F$, so does $h_1(\lambda)$. 
In view of Lemma~\ref{l3.1}, $h_1(t)\in F[t]$, which implies that $u(t)=\frac{h_1(t)}{h_2(t)} \in F(t)\subseteq F((t))$. Hence, ${f_{{i_0}}}\left( {{c_1},{c_2},\dots,{c_m}} \right)^\alpha\in F$. Thus, we have shown that $A$ satisfies a generalized polynomial identity $$f_{i_0}(y_1,y_2,\dots,y_m)^\alpha y-yf_{i_0}(y_1,y_2,\dots,y_m)^\alpha=0.$$ By Amitsur's Theorem (see \cite[Theorem~9.A, p. 282]{cohn}), $D$ is centrally finite  . The claim is proved.
		
Now, since $F$ is uncountable, $F$ is not algebraic over its prime subfield \cite[Corollary B-2.41, p. 343]{rotman}. Hence, by applying Theorem~\ref{Th2.9},  $N$ is central, a contradiction. The proof is now complete.
\end{proof}

\section{Free subgroups in almost subnormal subgroups in a division ring with uncountable center}

In this section, we focus on the problem of the existence of non-cyclic free subgroups in almost subnormal subgroups of division rings with uncountable centers. The first result on the existence of non-cyclic free subgroups in division rings with uncountable centers was obtained in \cite{rei-von}, where it was proved  that if $D$ is a division ring with uncountable center containing a non-central algebraic over the center element $a$, then $D^*$ contains a non-cyclic free subgroup. Later, Chiba \cite{chiba} proved the same result but  without the assumption of the existence of such an algebraic element. It seems to be impossible to use the techniques in \cite{chiba} to extend this result to all non-central almost subnormal subgroups of $D^*$ (see \cite{Pa_GoSh_09}). Probably, Gon\c calves and Shirvani were the first people who considered the existence of non-cyclic free subgroups in normal subgroups of $\GL_n(D)$, where $D$ is a division ring with uncountable center \cite{Pa_GoSh_09}. Recall that if $n\ge 2$, then a non-central almost subnormal subgroup of $\GL_n(D)$ is normal \cite[Theorem 3.3]{Pa_NgBiHa_2016}, and the problem on the existence of non-cyclic free subgroups in  normal subgroups of $\GL_n(D)$ for $n\ge 2$ was solved in the affirmative  in \cite{Pa_GoSh_09}, so in this section, we consider only the case  $n=1$.  In fact, we will show that in a division ring $D$ with uncountable center $F$, if a non-central almost subnormal subgroup of $D^*$ contains a non-central algebraic element over $F$, then it contains a non-cyclic free subgroup.

\bigskip

To prove the next theorem, we borrow the following results.
\begin{lemma}\label{BD}{\rm \cite[Proposition 2.2]{Pa_BiDu_14}} Let $D$ be a division ring with center $F$. If $a\in D\backslash F$ such that $a^n=1$, then there exists a centrally finite division ring $K$ such that $a$ is non-central in $K$.
\end{lemma}

\begin{lemma}\label{l7.2}{\rm \cite[Corollary 1.4]{Pa_GoSh_09}}  Let $D$ be a division ring with uncountable center $F$ and assume that $u\in D^*$ is an algebraic element over $F$. If $u\notin F$, then there exists $v\in D^*$ of infinite order such that the subgroup $\langle u, v\rangle$ of $D^*$ is isomorphic to the free product $\langle u\rangle*\langle v\rangle$.
\end{lemma}

\begin{theorem}\label{t7.2}
	Let $D$ be a division ring with uncountable center $F$ and assume that $G$ is a non-central almost subnormal subgroup of $D^*$. If $G\backslash F$ contains an algebraic element over $F$, then $G$ contains a non-cyclic free subgroup.
\end{theorem}
\begin{proof}  Assume that $u\in G\backslash F$ and $u$ is algebraic over $F$. We consider two cases:
	\bigskip
	
	\noindent
	\textit{Case 1. The order of $u$ is finite:} 
	
	\bigskip
	
	Then, by Lemma~\ref{BD}, there exists a centrally finite division ring $K$ of $D$ such that $u$ is a non-central element in $K$. Put $H=G\cap K^*$. Then, $H$ is almost subnormal in $K^*$ (\cite[Lemma 1.1]{Pa_NgBiHa_2016}) and $H$ is non-central since $u\in H$. Hence, $H$ contains a non-cyclic free subgroup by \cite[Theorem 4.2]{Pa_NgBiHa_2016}, so does $G$. 
	
	\bigskip
	
	\noindent
	\textit{ Case 2. The order of $u$ is infinite:}
	
	\bigskip
		
	By Lemma~\ref{l7.2}, there exists an element $v\in D^*$ of infinite order such that the subgroup $\langle u, v\rangle$ of $D^*$ is isomorphic to the free product $\langle u\rangle*\langle v\rangle$. Being the free product of two infinite cyclic groups, the group $\langle u\rangle*\langle v\rangle$ is free, so is $\langle u, v\rangle$. Let 
	$$G=G_r\leq G_{r-1}\leq \cdots\leq G_1=D^*$$
	be an almost normal series  in $D^*$. In view of Lemma~\ref{l7.1}, the elements $x=c_r(u,v)$ and $y=c_r(u,v^2)$ both belong to the group $G$. Now, let us recall that $$x=c_r(u,v)=uv^{m_1} u^{n_2}v^{m_2}\cdots u^{n_t}v^{m_t}, \mbox{ where } n_j,m_j \in \{1,-1\},$$ and 
	$$y=c_r(u,v^2)=uv^{m'_1} u^{n'_2}v^{m'_2}\cdots u^{n'_t}v^{m'_t}, \mbox{ where } n'_i\in \{1,-1\},  m'_i=2m_i,$$
	so  $xy\ne yx$. Being a non-abelian subgroup of the free group $\langle u, v\rangle$, in  view of the Nielsen-Schreier Theorem (see \cite[p. 103]{Pa_St_78}), the subgroup $\langle x,y\rangle$ is also a non-abelian free group.
	
	We see that in both cases, the group $G$ contains a non-cyclic free subgroup. Hence, the proof of the theorem is now complete.
	\end{proof}
	
	\noindent
	
	\textbf{Acknowledgements}

	The authors would like to express their sincere gratitude to the Editor and also to the referee for his/her careful reading, very useful comments and suggestions.

\end{document}